\documentclass[11pt, a4paper]{amsart}
\usepackage{amsthm, dsfont, a4}
\usepackage{hyperref}
\usepackage[all]{xy}
\usepackage[active]{srcltx}
\usepackage[short,nodayofweek]{datetime}

\newcommand\FF{{\mathbb F}}
\newcommand\GG{{\mathbb G}}
\newcommand\NN{{\mathbb N}}
\newcommand\ZZ{{\mathbb Z}}
\newcommand\CC{{\mathbb C}}
\newcommand\TT{{\mathbb T}}

\newcommand{\dumot}{{\mathfrak{M}}}
\newcommand{\mot}{{\mathsf{M}}}

\newcommand\pfrak{{\mathfrak p}}

\newcommand\pitilde{\tilde{\pi}}
\newcommand\id{{\rm id}}

\newcommand{\one}{\mathds{1}}
\newcommand{\transp}[1]{#1^{\rm tr}}   

\def\pr{{\rm pr}}
\newcommand{\hd}[2]{\partial_t^{(#1)}\!\!\left(#2\right)}
\newcommand{\hde}[1]{\partial_t^{(#1)}}  

\newcommand{\tate}{\CC_\infty\langle t\rangle}  
\newcommand{\laurent}{\CC_\infty(\!(t)\!)}  
\newcommand{\powser}{\CC_\infty[\![t]\!]}  
\newcommand{\ps}[1]{[\![#1]\!]}  
\newcommand{\ls}[1]{(\!(#1)\!)}

\def\vect#1{\text{\boldmath $#1$\unboldmath}} 
\def\isom{\cong}

\DeclareMathOperator{\Hom}{Hom}
\DeclareMathOperator{\End}{End}
\DeclareMathOperator{\ch}{char}
\DeclareMathOperator{\Mat}{Mat}
\DeclareMathOperator{\GL}{GL}
\DeclareMathOperator{\diag}{diag}

\DeclareMathOperator{\Lie}{Lie}

\def\markdef{\bf }

\theoremstyle{plain}
\newtheorem{thm}{Theorem}[section]
\newtheorem{cor}[thm]{Corollary}
\newtheorem{lem}[thm]{Lemma}
\newtheorem{prop}[thm]{Proposition}

\theoremstyle{definition}
\newtheorem{defn}[thm]{Definition}

\newtheorem{rem}[thm]{Remark}

\begin{document}

\title[Prolongations and algebraic independence]{Prolongations of $t$-motives and algebraic independence of periods}
\author{Andreas Maurischat}
\address{\rm {\bf Andreas Maurischat}, Lehrstuhl A f\"ur Mathematik, RWTH Aachen University, Germany }
\email{\sf andreas.maurischat@matha.rwth-aachen.de}



\begin{abstract}
In this article we show that the coordinates of a period lattice generator of the $n$-th tensor
power of the Carlitz module are algebraically independent, if $n$ is prime to the characteristic. 
The main part of the paper, however, is devoted to a general construction for $t$-motives which we call \textit{prolongation}, and which gives the necessary background for our proof of the algebraic independence. Another ingredient is a theorem which shows hypertranscendence for the Anderson-Thakur function $\omega(t)$, i.e.~that $\omega(t)$ and all its hyperderivatives with respect to $t$ are algebraically independent.
\end{abstract}

\maketitle

\setcounter{tocdepth}{1}
\tableofcontents

\section{Introduction}

Periods of $t$-modules play a central role in number theory in positive characteristic, and
questions about their algebraic independence are of major interest.
The most prominent period is the Carlitz period
\[ \tilde{\pi}=\lambda_\theta \theta \prod_{j \geq 1} (1 - \theta^{1-q^j})^{-1} \in K_\infty(\lambda_\theta), \]
where $\lambda_\theta\in \bar{K}$ is a $(q-1)$-th root of $-\theta$.
Here, $K=\FF_q(\theta)$ is the rational function field over the finite field $\FF_q$, $\bar{K}$ its algebraic closure, and $K_\infty$ is the completion of $K$ with respect to the absolute value
$|\cdot|_\infty$ given by $|\theta|_\infty=q$. 

The Carlitz period is the function field analog of the complex number $2\pi i$, and it was already proven 
by Wade in 1941 that $\tilde{\pi}$ is transcendental over $K$ (see~\cite{liw:cqtg}).

For proving algebraic independence of periods (and other ``numbers'' like zeta values and logarithms) the ABP-criterion (cf.~\cite[Thm.~3.1.1]{ga-wb-mp:darasgvpc}) and a consequence of it - which is part of the proof of \cite[Thm.~5.2.2]{mp:tdadmaicl} - turned out to be very useful. To state this consequence, let $\CC_\infty$ denote the completion of the algebraic closure of $K_\infty$, and
$\powser$ the power series ring over $\CC_\infty$, as well as 
$\TT=\tate$ the subring consisting of those power series which converge on the closed unit disc $|t|_\infty\leq 1$. 
Finally, let $\mathbb{E}$ be the subring of entire functions, i.e.~of those power series which converge for all $t\in \CC_\infty$ and whose coefficients lie in a finite extension of $K_\infty$.
On $\TT$ we consider the inverse Frobenius twist $\sigma$ given by
\[ \sigma( \sum_{i=0}^\infty x_it^i)=\sum_{i=0}^\infty (x_i)^{1/q}t^i, \]
which will be applied on matrices entry-wise.

\begin{thm} (See proof of \cite[Thm.~5.2.2]{mp:tdadmaicl})  \label{thm:conseq-of-abp}
\footnote{Note that the difference equation in \cite{mp:tdadmaicl} is given as $\sigma(\Psi)=\Phi\Psi$ from which our version is obtained by transposing the matrices. We use this transposed version as it fits better to our convention on notation (cf.~Sect.~\ref{subsec:conventions}).}

Let $\Phi\in \Mat_{r\times r}(\bar{K}[t])$ be a matrix with
determinant $\det(\Phi)=c(t-\theta)^s$ for some $c\in \bar{K}^\times$ and $s\geq 1$.
If $\Psi\in \GL_r(\TT)\cap \Mat_{r\times r}(\mathbb{E})$ is a matrix such that
\[  \sigma(\Psi)=\Psi \Phi, \]
then the transcendence degree of $\bar{K}(t)(\Psi)$ over $\bar{K}(t)$ is the same as the 
transcendence degree of $\bar{K}(\Psi(\theta))$ over $\bar{K}$.\\
Here, $\bar{K}(t)(\Psi)$ denotes the field extension of $\bar{K}(t)$ generated by the entries of
$\Psi$, and $\bar{K}(\Psi(\theta))$ denotes the field extension of $\bar{K}$ generated by the entries of
$\Psi(\theta)$, the evaluation of the entries of $\Psi$ at $t=\theta$.
\end{thm}

Actually, the matrix $\Phi$ occurs as a matrix which represents the $\sigma$-action on a dual $t$-motive $\dumot$ with respect to some $\bar{K}[t]$-basis of $\dumot$, and $\Psi$ is the corresponding rigid analytic trivialization.

Using this statement, one can also reprove the transcendence of $\pitilde$ by using the power series
\[  \Omega(t)= \lambda_\theta^{-q} \prod_{j \geq 1} (1 - \frac{t}{\theta^{q^j}}) \in \mathbb{E}. \]
This power series satisfies the difference equation $\sigma(\Omega)=\Omega\cdot  (t-\theta)$ and is indeed the
rigid analytic trivialization of the dual Carlitz motive $\mathfrak{C}$.
The function $ \Omega$ is transcendental over $\bar{K}(t)$ - as it has infinitely many zeros - and
\[ \Omega(\theta)=\Omega|_{t=\theta}=  \lambda_\theta^{-q} \prod_{j \geq 1} (1 - \frac{\theta}{\theta^{q^j}})=-\frac{1}{\theta \lambda_\theta} \prod_{j \geq 1} (1 - \theta^{1-q^j})
=-\frac{1}{\pitilde}.
\]
Hence by the criterion, $\pitilde$ is transcendental over $\bar{K}$.

Several proofs on algebraic independence (see e.g.~\cite{cc-mp:aipldm},\cite{ym:aicppcmv},\cite{mp:tdadmaicl}) follow the strategy to construct dual $t$-motives such that for the rigid analytic trivialization $\Psi$ of this module, the inverse of its specialization $\Psi(\theta)^{-1}$ has the desired values as entries. Then one shows algebraic independence for the corresponding entries of $\Psi$ or $\Psi^{-1}$ using different methods (like the Galois theoretical methods developed in  \cite{mp:tdadmaicl}) and deduces algebraic independence of the desired values.

\medskip

The main theorem in the present paper is about the periods of the $n$-th tensor power of the Carlitz module.
The $n$-th tensor power $E=C^{\otimes n}$ of the Carlitz module is a uniformizable $t$-module of dimension $n$ and rank $1$. Hence, the period lattice for $E$ is an $\FF_q[\theta]$-submodule of 
$\Lie(E)(\CC_\infty)\isom \CC_\infty^n$ of rank $1$, and we will show the following.

\begin{thm}  (see Thm.~\ref{thm:algebraic-independence})\\
Let $n\in\NN$ be prime to $q$, let $C^{\otimes n}$ be the $n$-th tensor power of the Carlitz module and let
\[ \begin{pmatrix} z_1 \\ \vdots \\ z_n \end{pmatrix} \in \CC_\infty^n\]
be a generator for the period lattice. Then $z_1,z_2,\ldots, z_n$ are algebraically independent over $\bar{K}$.
\end{thm}

The first step will be the definition of an appropriate dual $t$-motive such that the specialization at $t=\theta$ of the inverse of the rigid analytic trivialization contains such coordinates $z_1,\ldots, z_n$.
As this is a special case of a general construction of new $t$-motives from old ones, we present this construction in detail. Actually, the main part of the paper is devoted to this construction which we call \textit{prolongation}, due to its similarities to prolongations in differential geometry.

In Section  \ref{sec:prolongations-of-t-motives}, we start by defining the prolongations of (non-dual) $t$-motives, since they are often defined over a smaller base field than the dual $t$-motives, and we show various properties which transfer from the original $t$-motive to its prolongation.
We also give the explicit descriptions with matrices for abelian $t$-motives.
In Section \ref{sec:prolongations-of-dual-t-motives}, we transfer the definition of prolongation and the explicit description to dual $t$-motives, and in Section \ref{sec:prolongations-of-t-modules}, we transfer it to $t$-modules, too.

\medskip 

For the definition of prolongations, we make use of hyperdifferential operators (also called iterative higher derivations) with respect to the variable $t$. These are the family of $\CC_\infty$-linear maps $(\hde{n})_{n\geq 0}$ given by
\[    \hd{n}{\sum_{i=i_0}^\infty x_it^i} =\sum_{i=i_0}^\infty \binom{i}{n} x_it^{i-n} \]
for Laurent series $\sum_{i=i_0}^\infty x_it^i\in \laurent$, 
where $\binom{i}{n}\in \FF_p\subset \FF_q$ is the residue of the usual binomial coefficient.
One should think of the $n$-th hyperdifferential operator $\hde{n}$ as $\frac{1}{n!}(d/dt)^n$, although in characteristic $p$, we can't divide by $n!$, if $n\geq p$. In characteristic zero, however, $\hde{n}$ would be
exactly  $\frac{1}{n!}(d/dt)^n$.

As a warning to the reader, we would like to note that in the literature usually the hyperdifferential operators with respect to $\theta\in K$ are used (e.g.~in \cite{db:lidddm-I}, \cite{db:meagli}, \cite{db-ld:lidddm-II}), and hence the operation on power series is by hyperdifferentiating the coefficients. In this article, we will not use those hyperdifferential operators, but exclusively the hyperdifferentiation by $t$.

\medskip

In the proof of the main theorem, the Anderson-Thakur function $\omega(t)$ and its hyperderivatives appear, as $\omega$ is related to $\Omega$ via
\[   \omega=\frac{1}{(t-\theta)\Omega}. \]
In Section \ref{sec:omega-hypertranscendental}, we show a  property of $\omega$ which is of interest on its own, namely we show
\begin{thm} (see Thm.~\ref{thm:hypertranscendence})
The Anderson-Thakur function $\omega(t)$ is hypertranscendental over $\bar{K}(t)$, i.e.~the set $\{\hd{n}{\omega} \mid n\geq 0\}$ is algebraically independent over $\bar{K}(t)$.
\end{thm}

This will be deduced from properties of specializations of $\omega$ and its hyperderivatives at roots of unity which were investigated in \cite{ba-fp:ugtsls} and \cite{am-rp:iddbcppte}. This statement has also been given in \cite{fp:vscce} whose proof uses different methods.

\medskip

\subsection*{Acknowledgement}
I would like to thank R.~Perkins who turned my attention to the work of Angl\'es-Pellarin \cite{ba-fp:ugtsls}, in which our common paper \cite{am-rp:iddbcppte} resulted, and
which marked the beginning of the investigations presented in this article. 
I would also like to thank F.~Pellarin for interesting discussions on the hypertranscendence of $\omega$.

\section{Generalities}

\subsection{Base rings and operators}

Let $\FF_q$ be the finite field with $q$ elements, and $K$ a finite extension of the rational function field $\FF_q(\theta)$ in the variable $\theta$. We choose an extension to $K$ of the absolute value $|\cdot |_\infty$ which is given on $\FF_q(\theta)$  by $|\theta|_\infty=q$.
 Furthermore, $K_\infty\supseteq \FF_q\ls{\frac{1}{\theta}}$ denotes the completion of $K$ at this infinite place, and
$\CC_\infty$ the completion of an algebraic closure of $K_\infty$. Furthermore, let $\bar{K}$ be the algebraic closure of $K$ inside $\CC_\infty$.

All the commutative rings occuring will be subrings of the field of Laurent series $\laurent$, like
the polynomial rings $K[t]$ and $\bar{K}[t]$, the power series ring $\powser$ and the
Tate algebra $\TT=\tate$, i.e.~the algebra of series which are convergent for $|t|_\infty\leq 1$. 

On $\laurent$ we have several operations which will induce operations on these subrings.

First of all, there is the twisting $\tau:\laurent\to \laurent$ given by 
\[  f^{\tau} 
:=\sum_{i=i_0}^\infty (x_i)^qt^i \]
for $f=\sum_{i=i_0}^\infty x_it^i\in \laurent$, and the inverse twisting $\sigma:\laurent\to \laurent$ given by 
\[  f^\sigma := \sum_{i=i_0}^\infty (x_i)^{1/q}t^i \]
for $f=\sum_{i=i_0}^\infty x_it^i\in \laurent$. While the twisting restricts to endomorphisms on all subrings of $\laurent$ which occur in this paper, the inverse twisting is only defined for perfect coefficient fields, in particular not on $K[t]$, but on $\bar{K}[t]$.

On the Laurent series ring $\laurent$ we furthermore have an action of the hyperdifferential operators with respect to $t$, i.e.~the sequence of $\CC_\infty$-linear maps $(\hde{n})_{n\geq 0}$ given by
\[    \hd{n}{\sum_{i=i_0}^\infty x_it^i} =\sum_{i=i_0}^\infty \binom{i}{n} x_it^{i-n}. \]
The image $\hd{n}{f}$ of some $f\in \laurent$ is called the $n$-th hyperderivative of $f$.

The hyperdifferential operators satisfy $\hd{0}{f}=f$ for all $f\in \laurent$,
\[ \hd{n}{fg}=\sum_{i=0}^n \hd{i}{f}\hd{n-i}{g} \quad \text{for all }f,g\in \laurent, n\in \NN \]
as well as
\[  \hd{n}{\hd{m}{f}}=\binom{n+m}{n}\hd{n+m}{f}\quad \text{for all }f\in \laurent, n,m\in \NN. \]

It is not hard to verify that the subrings $\powser$, $\TT$, $L[t]$, and $L(t)$ (for any subfield
$L$ of $\CC_\infty$) are stable under all the hyperdifferential operators. It is also obvious that the hyperdifferential operators commute with the twistings $\tau$ and~$\sigma$.

Another way to obtain these hyperdifferential operators is to consider the $\CC_\infty$-algebra
map $\mathcal{D} : \powser \rightarrow \powser[\![X]\!]$
\[f(t) \mapsto f(t+X) = \sum_{n \geq 0} f_n(t) X^n \]
given by replacing the variable $t$ in the power series expansion for $f$ by $t+X$, expanding each $(t+X)^n$ using the binomial theorem, and rearranging to obtain a power series in $X$. Then, one has
 \[\hd{n}{f}=f_n. \]
Since, $\hd{0}{f}=f$ for all $f\in  \powser$, the homomorphism $\mathcal{D}$ can be extended  to a $\CC_\infty$-algebra map $\mathcal{D} : \laurent \rightarrow \laurent[\![X]\!]$, and we still have the identity
\[  \mathcal{D}(f)= \sum_{n \geq 0} \hd{n}{f} X^n. \]
For more background on hyperdifferential operators (iterative higher derivations) see for example \cite[\S 27]{hm:crt}.
\footnote{As already mentioned in the introduction, these hyperdifferential operators are not the one commonly used for constructing $t$-modules.
}

\medskip

When we apply the twisting operators $\tau$ and $\sigma$ as well as the hyperdifferential operators to matrices it is meant that we apply them entry-wise.

\medskip

We will frequently use the  following family (in $n \geq 0$) of homomorphisms of $\CC_\infty$-algebras $\rho_{[n]} : \laurent \rightarrow \Mat_{(n+1) \times (n+1)}(\laurent)$ defined by
\begin{equation}\label{eq:rho_n}
 \rho_{[n]}(f) := \begin{pmatrix} f & \hd{1}{f} & \cdots & \hd{n}{f} \\ 0 & f & \ddots & \vdots \\ \vdots & \ddots & \ddots & \hd{1}{f} \\ 0 & \cdots & 0 & f \end{pmatrix},
\end{equation}
which already appears in \cite{am-rp:iddbcppte}. This map 
arises from the homomorphism $\mathcal{D}$ by evaluation of $X$ at the $(n+1) \times (n+1)$ nilpotent matrix 
\[ N= \begin{pmatrix}
 0 & 1 & 0 & \cdots & 0 \\ \vdots & \ddots & \ddots & \ddots & \vdots  \\ \vdots && \ddots & \ddots & 0 \\ \vdots && & \ddots & 1 \\
 0 &  \cdots &\cdots &\cdots & 0
\end{pmatrix}. \]

We will also apply $\rho_{[n]}$ to square matrices $\Theta\in \Mat_{r\times r}(\laurent)$. In that case, $\rho_{[n]}(\Theta)$ is defined to be the block square matrix
\begin{equation}\label{eq:rho_n-matrix} \rho_{[n]}(\Theta):= \begin{pmatrix}
\Theta & \hd{1}{\Theta}  & \hd{2}{\Theta} & \cdots& \hd{n}{\Theta} \\
0 & \Theta &  \hd{1}{\Theta} &  \ddots   & \vdots \\ 
\vdots &\ddots  & \ddots & \ddots &   \hd{2}{\Theta}\\
\vdots & & \ddots  &  \Theta &  \hd{1}{\Theta} \\
0 &  \cdots & \cdots & 0 &  \Theta
\end{pmatrix} 
\end{equation}
in the ring of $r(n+1)\times r(n+1)$-matrices. As mentioned before $\hd{1}{\Theta}$ etc. is
the matrix where we apply the hyperdifferential operators coefficient-wise.
It is not hard to check that $\rho_{[n]}:\Mat_{r\times r}(\laurent)\to \Mat_{r(n+1)\times r(n+1)}(\laurent)$
is a ring homomorphism, too.

As the hyperdifferential operators commute with twisting, $\rho_{[n]}$ also commutes with twisting.

\subsection{Convention on notation}\label{subsec:conventions}

In the following sections, we will deal with $t$-modules, $t$-motives and dual $t$-motives. We use the definitions of these terms as given 
 in the survey article \cite{db-mp:ridmtt}. For the convenience of the reader, we repeat these definitions below, but refer the reader to ibid. for more details.
For recognizing the objects at first glance, $t$-modules will be denoted by italic letters, like $E$, $t$-motives with serif-less letters, like $\mot$, and dual $t$-motives in Fraktur font, like~$\dumot$.

Bases of 
finitely generated free modules (over some ring) will always be written as row vectors $\vect{e}=(e_1,\ldots, e_r)$
such that one obtains the familiar identification of the module with a module of column vectors by
writing an arbitrary element $x=\sum_{i=1}^r x_ie_i$ as
\[   \vect{e}\cdot \begin{pmatrix}
x_1\\ \vdots \\ x_r \end{pmatrix}. \]

A {\markdef $t$-module} $(E,\Phi)$ (or shortly, $E$) consists of an algebraic group $E$ over $K$ which is isomorphic to $\GG_a^d$ for some $d>0$, and an $\FF_q$-algebra homomorphism
\[ \Phi:\FF_q[t]\to \End_{{\rm grp},\FF_q}(E)\isom \Mat_{d\times d}(K\{\tau\}), \]
with the additional property that $\Phi(t)-\theta\cdot \id_E$ induces a nilpotent endomorphism on $\Lie(E)$.
In other terms, if one writes
\[ \Phi(t)= A_0+A_1\tau+\ldots +A_s\tau^s\in  \Mat_{d\times d}(K\{\tau\}) \]
with respect to some isomorphism $\End_{{\rm grp},\FF_q}(E)\isom \Mat_{d\times d}(K\{\tau\})$, 
 then the matrix $A_0-\theta\cdot \one_d\in \Mat_{d\times d}(K)$ is nilpotent.

\medskip

A {\markdef $t$-motive} $\mot$ 
is a left $K[t]\{\tau\}$-module which is free and finitely generated as $K\{\tau\}$-module, and such that
\[  (t-\theta)^\ell(\mot)\subseteq K[t]\cdot \tau(\mot) \]
for some $\ell\in \NN$. A $t$-motive $\mot$ is called {\markdef abelian} if it is also finitely generated as $K[t]$-module in which case it is even free as $K[t]$-module. An abelian $t$-motive $\mot$ is called {\markdef pure}, if there exists a $K\ps{1/t}$-lattice $H$ inside $\mot\otimes_{K[t]} K\ls{1/t}$ and $u,v\geq 1$ such that
\[  t^uH= K\ps{1/t}\cdot \tau^v H. \]
The fraction $w=\frac{u}{v}$ is called the {\markdef weight} of $\mot$.

Given an abelian $t$-motive $\mot$ with $K[t]$-basis $\vect{e}=(e_1,\ldots, e_r)$, then there is a matrix $\Theta\in \Mat_{r\times r}(K[t])$ representing the $\tau$-action on $\mot$ with respect to $\{e_1,\ldots, e_r\}$, i.e.~
\[   \tau( e_j)= \sum_{h=1}^r \Theta_{hj}e_h \]
for all $j=1,\ldots, r$. This will be written in matrix notation as
\[  \tau (\vect{e})= \vect{e}\cdot \Theta.\]
For an arbitrary element $x=\sum_{i=1}^r x_ie_i$ one therefore has
\[ \tau(x)=\vect{e}\cdot \Theta \cdot \begin{pmatrix}
x_1\\ \vdots \\ x_r \end{pmatrix}^\tau. \]

\medskip

Writing the basis as a row vector instead of a column vector, as for example in \cite{mp:tdadmaicl}, causes the difference equations for the rigid analytic trivializations to have a different form which we will review now. However, the usual form is obtained by taking transposes of the matrices given here:

Given an abelian $t$-motive $\mot$ with $K[t]$-basis $\vect{e}=(e_1,\ldots, e_r)$ and
$\Theta\in \Mat_{r\times r}(K[t])$ such that
\[  \tau (\vect{e})= \vect{e}\cdot \Theta,\]
a {\markdef rigid analytic trivialization} (if it exists) is a matrix $\Upsilon\in \GL_{r}(\TT)$
such that $\tau(\vect{e}\cdot \Upsilon)=\vect{e}\cdot \Upsilon$, i.e.~such that
\[  \Theta\cdot \Upsilon^\tau=\Upsilon. \]
If $\Upsilon$ exists, $\mot$ is called {\markdef rigid analytically trivial}.

In \cite{ga:tm}, Anderson associated to a $t$-module $E$ a $t$-motive 
$\mathsf{E}:=\Hom_{{\rm grp},\FF_q}(E,\GG_a)$ with $t$-action given by composition with 
$\Phi_t\in \End_{{\rm grp},\FF_q}(E)$ and left-$K\{\tau\}$-action given by composition
with elements in $K\{\tau\}\isom \End_{{\rm grp},\FF_q}(\GG_a)$. A $t$-module is then called {\markdef abelian} if the associated $t$-motive is abelian,
 and Anderson proved (cf.~\cite[Thm.~1]{ga:tm}) that this correspondence
induces an anti-equivalence of categories between abelian $t$-modules and abelian $t$-motives. However, the proof even shows that it induces an anti-equivalence of categories between $t$-modules and $t$-motives.

\medskip

A {\markdef dual $t$-motive} $\dumot$ is a left $\bar{K}[t]\{\sigma\}$-module that is free and finitely generated as $\bar{K}\{\sigma\}$-module,  and such that
\[  (t-\theta)^\ell(\dumot)\subseteq \sigma(\dumot) \]
for some $\ell\in \NN$.  A dual $t$-motive is called {\markdef $t$-finite} if it is also finitely generated as $\bar{K}[t]$-module in which case it is even free as $\bar{K}[t]$-module.

For a $t$-finite dual $t$-motive $\dumot$ with $K[t]$-basis $\vect{e}=(e_1,\ldots, e_r)$ and
$\tilde{\Theta}\in \Mat_{r\times r}(K[t])$ such that
\[  \sigma (\vect{e})= \vect{e}\cdot \tilde{\Theta}\]
a {\markdef rigid analytic trivialization} (if it exists) is a matrix $\Psi\in \GL_{r}(\TT)$ such that $\sigma(\vect{e}\cdot \Psi^{-1})=\vect{e}\cdot \Psi^{-1}$, i.e.~such that
\[  \Psi\cdot \tilde{\Theta}=\Psi^\sigma. \]
If $\Psi$ exists, $\dumot$ is called {\markdef rigid analytically trivial}.

Similar, as for $t$-motives, Anderson associated to a $t$-module $E$ over $\bar{K}$ a dual $t$-motive
$\mathfrak{E}:=\Hom_{{\rm grp},\FF_q}(\GG_a,E)$ with $t$-action given by composition with 
$\Phi_t\in \End_{{\rm grp},\FF_q}(E)$ and left-$K\{\sigma\}$-action given by composition
with elements in $K\{\sigma\}\isom K\{\tau\}^{\rm op}\isom \End_{{\rm grp},\FF_q}(\GG_a)^{\rm op}$.
Anderson showed (cf.~\cite{uh-akj:pthshcff}) that this induces an equivalence of categories between $t$-modules over $\bar{K}$ and $t$-motives.

\section{Prolongations of $t$-motives}\label{sec:prolongations-of-t-motives}

In this section, we introduce a construction of new $t$-motives from old ones which we call
\textit{prolongation}. The construction is taken from \cite{mk:tffo} where prolongations of difference modules are described. We also show (see Theorems \ref{thm:prolongation-motive-abelian} and \ref{thm:pure-and-r-a-t}) that the prolongations inherit the properties of abelianness, rigid analytic triviality as well as pureness from the original $t$-motive.

\begin{defn}\label{def:prolongation}
For a $K[t]\{\tau\}$-module $\mot$ and $k\geq 0$, the {\markdef $k$-th prolongation} of 
$\mot$ is the $K[t]$-module $\rho_k\mot$ which is generated by symbols
$D_im$, for $i=0,\ldots, k$, $m\in \mot$, subject to the relations
\begin{enumerate}
\item $D_i(m_1+m_2)=D_im_1+D_im_2$, 
\item\label{item:second-relation} $D_i(a\cdot m)=\sum_{i_1+i_2=i} \hd{i_1}{a}\cdot D_{i_2}m$,
\end{enumerate}
 for all $m,m_1,m_2\in \mot$, $a\in K[t]$ and $i=0,\ldots, k$.
 The semi-linear $\tau$-action on $\rho_k\mot$ is given by
 \[   \tau( a\cdot D_im)=a^{\tau}\cdot D_i(\tau(m)). \]
 for  $a\in K[t]$, $m\in \mot$.
\end{defn}

One should think of $D_im$ as being the formal $i$-th hyperderivative of the element~$m$.

\begin{rem}\label{rem:prolongation-is-extension}
It is not difficult to verify that the definition of the $\tau$-action is well-defined.
Hence, the $k$-th prolongation $\rho_k\mot$ is again a $K[t]\{\tau\}$-module.

Furthermore, $\rho_0\mot$ is naturally isomorphic to $\mot$ (via $D_0m\mapsto m$), and
for $0\leq l<k$ the $l$-th prolongation  $\rho_l\mot$ naturally is a  $K[t]\{\tau\}$-submodule of $\rho_k\mot$. For  $0\leq l<k$, we even obtain a short exact sequence of $K[t]\{\tau\}$-modules
\[  0\longrightarrow \rho_l\mot \longrightarrow \rho_k\mot \xrightarrow{\pr} \rho_{k-l-1}\mot \to 0 \]
where $\pr(D_im):= D_{i-l-1}m$ for $i>l$ and all $m\in \mot$, as well as
$\pr(D_im):=0$ for $i\leq l$ and all $m\in \mot$.
In particular, taking $l=k-1$ and using the identification $\rho_0\mot\isom\mot$, we obtain the short exact sequence
\begin{equation}\label{eq:short-exact-sequence}
 0\longrightarrow \rho_{k-1}\mot \longrightarrow \rho_k\mot \longrightarrow \mot \to 0 . \tag{*}
 \end{equation} 
Inductively, we see that $\rho_k\mot$ is a $(k+1)$-fold extension of $\mot$ with itself.

From this description as a $(k+1)$-fold extension of $\mot$ with itself, we will be able to transfer several additional properties of $\mot$ to the prolongation $\rho_k\mot$ (see Theorem \ref{thm:prolongation-motive-abelian} and Theorem \ref{thm:pure-and-r-a-t}).
\end{rem}

\begin{lem}\label{lem:prolongation-as-K-v-s}
As a $K$-vector space, the $k$-th prolongation $\rho_k\mot$ is generated by the symbols
$D_im$, for $i=0,\ldots, k$, $m\in \mot$, subject to the relations
\[  D_i(x_1m_1+x_2m_2)= x_1\cdot D_im_1+x_2\cdot D_im_2 \]
for all $m_1,m_2\in \mot$, $x_1,x_2\in K$ and $i=0,\ldots, k$.
The actions of $t$ and $\tau$ are described by
\begin{eqnarray*}
t\cdot D_im &=& D_i(tm)-D_{i-1}m \\
\tau( D_im ) &=& D_i(\tau(m))
\end{eqnarray*}
for $m\in \mot$, $i=0,\ldots, k$ where we set $D_{-1}m:=0$.
\end{lem}

\begin{proof}
Applying relation \eqref{item:second-relation} above to $a=t$, leads to
\[ D_i(tm)= t\cdot D_im + 1\cdot D_{i-1}m \]
for all $m\in \mot$. Hence, $ t\cdot D_im=D_i(tm)-D_{i-1}m$.\\
This shows that $K[t]$-multiples of the $D_im$ are in the $K$-span of all $D_i(m')$, and therefore
$\rho_k\mot$ is generated by all  $D_im$ as a $K$-vector space.

Restricting relation \eqref{item:second-relation} to $a\in K$, we obtain $D_i(a\cdot m)=a\cdot D_im$
for all $a\in K$ and $m\in \mot$. Hence, the relations above reduce to
\[  D_i(x_1m_1+x_2m_2)= x_1\cdot D_im_1+x_2\cdot D_im_2 \]
for all $m_1,m_2\in \mot$, $x_1,x_2\in K$ and $i=0,\ldots, k$.

The given actions are clear from the equation above and the definition of~$\rho_k\mot$.
\end{proof}

\begin{thm}\label{thm:prolongation-motive-abelian}
Let $\mot$ be a $t$-motive. Then the $k$-th prolongation  $\rho_k\mot$ is a $t$-motive
 for all $k\geq 0$.

If  $\mot$ is abelian, then so is $\rho_k\mot$.
\end{thm}

\begin{proof}
By Remark \ref{rem:prolongation-is-extension}, we have an exact sequence 
of $K[t]\{\tau\}$-modules
\[
 0\longrightarrow \rho_{k-1}\mot \longrightarrow \rho_k\mot \longrightarrow \mot \to 0 
\]
using the identification $\rho_0\mot\isom\mot$ (see Equation \eqref{eq:short-exact-sequence}).
Hence, it follows by induction on $k$ that $\rho_k\mot$ is free and finitely generated as $K\{\tau\}$-module if $\mot$ is. Furthermore, if $\ell\in \NN$ is such that
\[  (t-\theta)^\ell (\mot)\subseteq K[t] \cdot \tau(\mot), \]
we obtain
\[  (t-\theta)^\ell (\rho_k\mot)\subseteq  K[t] \cdot \tau (\rho_k\mot)+ \rho_{k-1}\mot, \]
and hence, inductively,
\[  (t-\theta)^{\ell\cdot (k+1)} (\rho_k\mot)\subseteq  K[t] \cdot \tau (\rho_k\mot).\]

Therefore, $\rho_k\mot$ is a $t$-motive.

If $\mot$ is abelian, i.e.~free and finitely generated as a $K[t]$-module, then $\rho_k\mot$ is free and finitely generated as a $K[t]$-module, since it is a $(k+1)$-fold extension of copies of~$\mot$.
\end{proof}

\begin{lem}\label{lem:k-tau-basis-of-prolongation}
Let $\mot$ be a $t$-motive, and $\vect{b}=(b_1,\ldots,b_d)$ be a $K\{\tau\}$-basis of $\mot$.
Then a $K\{\tau\}$-basis of $\rho_k\mot$ is given by
\[  \vect{Db}=(D_0b_1,\ldots, D_0b_d, D_1b_1,\ldots, D_1b_d,\ldots,\ldots, D_kb_1,\ldots, D_kb_d).\]
\end{lem}

\begin{proof}
From the short exact sequence \eqref{eq:short-exact-sequence} we see that
a $K\{\tau\}$-basis of $\rho_k\mot$ is given by the join of a $K\{\tau\}$-basis of $\rho_{k-1}\mot$ and the preimage of a basis of $\mot$. As such a preimage is given by $(D_kb_1,\ldots, D_kb_d)$ the proof follows by induction.
\end{proof}

We are now going to explicitly describe the $t$-motive $\rho_k\mot$ as $K[t]$-module with $\tau$-action in the abelian case, i.e.~we give a basis as $K[t]$-module as well as a matrix representation of the $\tau$-action with respect to this $K[t]$-basis.

\smallskip

Assume that $\mot$ is an abelian $t$-motive, and let $\vect{e}=(e_1,\ldots, e_r)$ be a $K[t]$-basis of 
$\mot$. As in the previous lemma, from the short exact sequence \eqref{eq:short-exact-sequence} in Remark \ref{rem:prolongation-is-extension} we obtain that
$\vect{De}=(D_0e_1,\ldots, D_0e_r, D_1e_1,\ldots, D_1e_r,\ldots,\ldots, D_ke_1,\ldots, D_ke_r)$ is
a $K[t]$-basis of $\rho_k\mot$.

Let $\Theta\in \Mat_{r\times r}(K[t])$ be the matrix representing the $\tau$-action on $\mot$ with respect to $\vect{e}=(e_1,\ldots, e_r)$, i.e.~
\[   \tau( e_j)= \sum_{h=1}^r \Theta_{hj}e_h \]
for all $j=1,\ldots, r$, or in matrix notation
\[  \tau (\vect{e})= \vect{e}\cdot \Theta.\]
Then $\tau$ acts on $D_ie_j\in \rho_k\mot$ as
\[  \tau(D_ie_j)= D_i(\tau(e_j))=D_i(\sum_{h=1}^r \Theta_{hj}e_h)
=\sum_{h=1}^r \sum_{i_1+i_2=i} \hd{i_1}{\Theta_{hj}}\cdot D_{i_2}e_h. \]
In block matrix notation this reads as
\[  \tau(\vect{De}) =  \vect{De}\cdot  \begin{pmatrix}
\Theta & \hd{1}{\Theta}  & \hd{2}{\Theta} & \cdots& \hd{k}{\Theta} \\
0 & \Theta &  \hd{1}{\Theta} &  \ddots   & \vdots \\ 
\vdots &\ddots  & \ddots & \ddots &   \hd{2}{\Theta}\\
\vdots & & \ddots  &  \Theta &  \hd{1}{\Theta} \\
0 &  \cdots & \cdots & 0 &  \Theta
\end{pmatrix} = \vect{De}\cdot \rho_{[k]}(\Theta), \]
where we use the homomorphism $\rho_{[k]}$ defined in Equation \eqref{eq:rho_n-matrix}.


\begin{thm}\label{thm:pure-and-r-a-t}
Let $\mot$ be an abelian $t$-motive, $k\geq 0$ and $\rho_k\mot$ the $k$-th prolongation
of $\mot$.
\begin{enumerate}

\item If $\mot$ is rigid analytically trivial, then  $\rho_k\mot$ is rigid analytically trivial.
\item If $\mot$ is pure of weight $w$, then  $\rho_k\mot$ is pure of weight $w$.
\end{enumerate}
\end{thm}

\begin{proof}
Let $\mot$ be given with respect to a basis $\vect{e}=(e_1,\ldots, e_r)$ by the
$\tau$-action
\[  \tau (\vect{e})= \vect{e}\cdot \Theta\]
for some $\Theta\in \Mat_{r\times r}(K[t])$.
Assume that $\mot$ is rigid analytically trivial, and that $\Upsilon\in \GL_r(\TT)$ is a
rigid analytic trivialization of $\mot$, i.e.~$\Upsilon$ satisfies the difference equation
\[    \Upsilon = \Theta \Upsilon^\tau .\]
Since twisting commutes with $\rho_{[k]}$ and $\rho_{[k]}$ is a ring homomorphism, we have
\[ \rho_{[k]}( \Theta)\left(\rho_{[k]}(\Upsilon )\right)^\tau=\rho_{[k]}( \Theta \Upsilon^\tau )
= \rho_{[k]}(\Upsilon ). \]
Since the $\tau$-action on $\rho_k\mot$ with respect to 
 $\vect{De}$ from above
 is given by $\tau(\vect{De}) =  \vect{De}\cdot \rho_{[k]}(\Theta)$, this just means that
 $ \rho_{[k]}(\Upsilon )\in \GL_{r(k+1)}(\TT)$ is a rigid analytic trivialization of $\rho_k\mot$.
 
Assume that $\mot$ is pure of weight $w$, and let $H$ be a $K\ps{1/t}$-lattice inside
$\mot\otimes_{K[t]} K\ls{1/t}$ such that
\[   t^u H=K\ps{1/t}\cdot \tau^vH \]
for appropriate $u,v\geq 1$.

After choosing a $K\ps{1/t}$-basis $\vect{b}=(b_1,\ldots, b_r)$ of $H$, we have
\[ \tau^v(\vect{b})=\vect{b}\cdot t^uA \]
for some $A\in \GL_r(K\ps{1/t})$.
By the explicit description of the $\tau$-action on $\rho_k\mot$, we therefore get
\begin{eqnarray*}
  \tau^v(\vect{Db})&=&\vect{Db}\cdot \rho_{[k]}(t^uA)\\
&=& \vect{Db}\cdot t^u  \begin{pmatrix}
A & t^{-u}\hd{1}{t^uA}  & t^{-u}\hd{2}{t^uA} & \cdots& t^{-u}\hd{k}{t^uA} \\
0 & A &  t^{-u}\hd{1}{t^uA} &  \ddots   & \vdots \\ 
\vdots &\ddots  & \ddots & \ddots &   t^{-u}\hd{2}{t^uA}\\
\vdots & & \ddots  & A &  t^{-u}\hd{1}{t^uA} \\
0 &  \cdots & \cdots & 0 &  A
\end{pmatrix}
.
\end{eqnarray*}
For Laurent series $f=\sum_{j=j_0}^\infty x_j t^{-j}$ in $1/t$ we have
\[ \hd{n}{f} =\sum_{j=j_0}^\infty \binom{-j}{n} x_j t^{-j-n}. \]
In particular, for any power series $f=\sum_{j=0}^\infty x_j t^{-j}\in K\ps{1/t}$ and $u\in \ZZ$,
\begin{eqnarray*}
  t^{-u}\cdot \hd{n}{t^uf} &=&  t^{-u}\cdot \hd{n}{\sum_{j=0}^\infty x_j t^{-j+u}}
= t^{-u}\cdot \sum_{j=0}^\infty \binom{-j+u}{n} x_j t^{-j+u-n}\\
&=& \sum_{j=0}^\infty \binom{-j+u}{n} x_j t^{-j-n}\in t^{-n}K\ps{1/t} \subseteq K\ps{1/t}.
\end{eqnarray*} 
Hence, the block upper triangular matrix above has entries in $ K\ps{1/t} $, and is moreover
invertible over $ K\ps{1/t}$, as $A$ is invertible.
Hence, by choosing $\rho_kH$ to be the $K\ps{1/t}$-lattice inside
$\rho_k\mot\otimes_{K[t]} K\ls{1/t}$ generated by $\vect{Db}$ we obtain
\[     K\ps{1/t}\cdot \tau^v(\rho_kH)= t^u\rho_kH. \]
Hence, $\rho_k\mot$ is pure of weight $\frac{u}{v}=w$. 
\end{proof}

\begin{rem}
Starting with a Drinfeld module, the associated $t$-motive is abelian, pure and rigid analytically trivial.
Hence, by taking its prolongations we obtain new abelian, pure and rigid analytically trivial $t$-motives of arbitrary dimension.
\end{rem}

\section{Prolongations of dual $t$-motives}\label{sec:prolongations-of-dual-t-motives}

Since we will use the dual $t$-motives in the proof in Section \ref{sec:algebraic-independence}, we review the construction and explicit descriptions in this case.

For the definition of a prolongation of a dual $t$-motive $\dumot$ we just transfer
the definition for the $t$-motives above.

\begin{defn}\label{def:prolongation-dual-motive}
For a dual $t$-motive $\dumot$ over $\bar{K}[t]$ and $k\geq 0$, the {\markdef $k$-th prolongation} of 
$\dumot$ is the $\bar{K}[t]$-module $\rho_k\dumot$ which is generated by symbols
$D_im$, for $i=0,\ldots, k$, $m\in \dumot$, subject to the relations
\begin{enumerate}
\item $D_i(m_1+m_2)=D_im_1+D_im_2$, 
\item $D_i(a\cdot m)=\sum_{i_1+i_2=i} \hd{i_1}{a}\cdot D_{i_2}m$,
\end{enumerate}
 for all $m,m_1,m_2\in \dumot$, $a\in \bar{K}[t]$ and $i=0,\ldots, k$.
 The semi-linear $\sigma$-action on $\rho_k\mot$ is given by
 \[   \sigma( a\cdot D_im)=a^{\sigma}\cdot D_i(\sigma(m)). \]
 for  $a\in \bar{K}[t]$, $m\in \dumot$.
\end{defn}

We obtain similar explicit descriptions as for abelian $t$-motives.

\begin{prop}
Let $\dumot$ be a $t$-finite dual $t$-motive with $\bar{K}[t]$-basis 
$\vect{e}=(e_1,\ldots, e_r)$ and $\tilde{\Theta}\in \Mat_{r\times r}(\bar{K}[t])$
the matrix such that
\[  \sigma(\vect{e})=\vect{e}\cdot \tilde{\Theta}. \]
Then $\vect{De}=(D_0e_1,\ldots, D_0e_r, D_1e_1,\ldots, D_1e_r,\ldots,\ldots, D_ke_1,\ldots, D_ke_r)$ is a $\bar{K}[t]$-basis of $\rho_k\dumot$ and
\[ \sigma(\vect{De})
=\vect{De}\cdot \rho_{[k]}(\tilde{\Theta}). \]
If $\dumot$ is rigid analytically trivial with rigid analytic trivialization $\Psi$, i.e.~
$\Psi^\sigma = \Psi\cdot \tilde{\Theta}$, then $\rho_k\dumot$ is rigid analytically trivial 
and $\rho_{[k]}(\Psi)$ is a rigid analytic trivialization with respect to $\vect{De}$.
\end{prop}

\begin{proof}
The proof is along the same lines as for $t$-motives. 
\end{proof}

\section{Prolongations of $t$-modules}\label{sec:prolongations-of-t-modules}

\begin{defn}
Let $(E,\Phi)$ be a $t$-module, and $\mathsf{E}$ the corresponding $t$-motive. Then we define the $k$-th prolongation $(\rho_kE,\rho_k\Phi)$ of $(E,\Phi)$
to be the $t$-module associated to $\rho_k\mathsf{E}$.
\end{defn}

\begin{thm}
Let $(E,\Phi)$ be a $t$-module of dimension $d$, and
\[   \Phi_t=A_0+A_1\tau+\ldots +A_s\tau^s \in \Mat_{d\times d}(K\{\tau\}) \]
with repect to some isomorphism $E\isom \GG_a^d$.

Then the  $k$-th prolongation $(\rho_kE,\rho_k\Phi)$ of $(E,\Phi)$ is of dimension $d(k+1)$
and $(\rho_k\Phi)_t$ is given in block diagonal form as
\[ (\rho_k\Phi)_t= \begin{pmatrix}
A_0 & 0 & \cdots &\cdots & 0 \\
-\one_d & \ddots & \ddots && \vdots \\
0 & \ddots & \ddots & \ddots& \vdots \\
 \vdots  & \ddots & \ddots & \ddots& 0\\
 0&\cdots &0 & -\one_d & A_0
\end{pmatrix}+ 
\diag(A_1)\tau+\ldots + \diag(A_s)\tau^s,
\]
where $\one_d$ is the $(d\times d)$-identity matrix, and $\diag(A_i)$ is the block diagonal matrix with
diagonal entries all equal to $A_i$ for $i=1,\ldots, s$. 
\end{thm}

\begin{proof}
Let $\vect{e}=(e_1,\ldots,e_d)$ be the basis of $E$ corresponding to the isomorphism 
$E\isom \GG_a^d$, and hence the $t$-action is given by
\[  t(\vect{e})=\vect{e}\cdot \Phi_t. \]
Then a $K\{\tau\}$-basis for the $t$-motive $\mathsf{E}$ is given by the dual basis 
$\vect{e^\vee}=(e_1^\vee,\ldots,e_d^\vee)$ and the $t$-action on $\mathsf{E}$ is given by
\[  t(\vect{e^\vee})=\vect{e^\vee}\cdot \transp{\Phi_t}. \]
By Lemma \ref{lem:k-tau-basis-of-prolongation}, a $K\{\tau\}$-basis  of $\rho_k\mathsf{E}$ is given
by
\[ \vect{De^\vee}=(D_0e_1^\vee,\ldots, D_0e_d^\vee, D_1e_1^\vee,\ldots, D_1e_d^\vee,\ldots,\ldots, D_ke_1^\vee,\ldots, D_ke_d^\vee),\]
and we have
\[  t(D_ie_j^\vee)= D_i(te_j^\vee) - D_{i-1}e_j^\vee \]
for $i=0,\ldots,k$ and $j=1,\ldots, d$, where we set $D_{-1}e_j^\vee=0$. In block matrix notation this is just
\[ t(\vect{De^\vee})=\vect{De^\vee}\cdot 
\begin{pmatrix}
\transp{\Phi_t} & -\one_d & 0 &  \cdots & 0 \\
0 & \transp{\Phi_t} &  \ddots &  \ddots &  \vdots \\
\vdots &  \ddots& \ddots & \ddots & 0 \\
\vdots & &  \ddots& \ddots & -\one_d\\
0 & \cdots& \cdots & 0&  \transp{\Phi_t} 
\end{pmatrix}.
\]
This finally shows that $\rho_kE$ is isomorphic to $\GG_a^{d(k+1)}$ with basis $\vect{De}$, the dual basis of $\vect{De^\vee}$, and the $t$-action is given by
\[ t(\vect{De})= \vect{De}\cdot \begin{pmatrix} \Phi_t & 0 & \cdots& \cdots & 0 \\
-\one_d & \Phi_t & \ddots & & \vdots \\
0 & \ddots & \ddots & \ddots &\vdots \\
\vdots &  \ddots& \ddots & \ddots & 0 \\
0 & \cdots& 0 & -\one_d &  \Phi_t 
\end{pmatrix}. \]
Hence,
\[ (\rho_k\Phi)_t= \begin{pmatrix}
A_0 & 0 & \cdots &\cdots & 0 \\
-\one_d & \ddots & \ddots && \vdots \\
0 & \ddots & \ddots & \ddots& \vdots \\
 \vdots  & \ddots & \ddots & \ddots& 0\\
 0&\cdots &0 & -\one_d & A_0
\end{pmatrix}+ 
\diag(A_1)\tau+\ldots + \diag(A_s)\tau^s.
\]

\end{proof}

\section{Prolongations of tensor powers of the Carlitz motive}\label{sec:carlitz-case}

In this section, we apply the constructions of prolongations to the tensor powers of the Carlitz module, the Carlitz motive, as well as the dual Carlitz motive.

Let us first recall the (dual) Carlitz motive and its tensor powers. The Carlitz module $(C,\phi)$ is given by $C\isom \GG_a$ and
\[ \phi:A\to \End(\GG_{a,K})=K\{\tau\}, f\mapsto \phi_f \]
given by  $\phi_t=\theta+\tau$.
The Carlitz motive $\textsf{C}=\Hom_K(C,\GG_a)\isom K\{\tau\}$ is also free of rank $1$ as $K[t]$-module, and with respect to the basis element $e=1\in K\{\tau\}\isom \textsf{C}$ the $\tau$-action is given by
$\tau(e)=e\cdot (t-\theta).$

The $n$-th tensor power of the Carlitz motive $\textsf{C}$ is the $K[t]$-module
\[ \textsf{C}^{\otimes n}=\underbrace{\textsf{C}\otimes_{K[t]}\ldots \otimes_{K[t]}\textsf{C}}_{n-\text{times}} \]
with diagonal $\tau$-action. I.e.~on the canonical basis element $e_{\otimes n}$, we have 
\[  \tau(e_{\otimes n})=e_{\otimes n}\cdot (t-\theta)^n. \]

Let $\omega\in \TT$
be the Anderson-Thakur function.
Then a rigid analytic trivialization for $\textsf{C}$ is given by $\frac{1}{\omega}$,
since $\omega$ satisfies the difference equation $\omega^\tau=(t-\theta)\omega$.
Hence, a rigid analytic trivialization for $\textsf{C}^{\otimes n}$ is given by
$\omega^{-n}$.

\bigskip

The dual Carlitz motive $\mathfrak{C}$ is the $\bar{K}[t]$-module of rank $1$ with $\sigma$-action given by
\[  \sigma(e)=e\cdot (t-\theta), \]
with respect to some basis element $e\in \mathfrak{C}$, and its $n$-th tensor power
$\mathfrak{C}^{\otimes n}$ has $\sigma$-action given by
\[ \sigma(e_{\otimes n})=e_{\otimes n}\cdot (t-\theta)^n. \]

The entire function $\Omega(t):=\frac{1}{(t-\theta)\omega(t)}$ is a rigid analytic trivialization 
of the Carlitz dual $t$-motive $\mathfrak{C}$, since
\[  \Omega^\sigma=\left(  \frac{1}{\omega^\tau}\right)^\sigma= \frac{1}{\omega}
= (t-\theta)\Omega. \]
Therefore, $\Omega(t)^n$ is a rigid analytic trivialization for the $n$-th tensor power $\mathfrak{C}^{\otimes n}$.

\begin{prop}\label{prop:prolong-tensor-power}
The $k$-th prolongation of the motive $\textsf{C}^{\otimes n}$ is the $K[t]$-module
$\rho_k(\textsf{C}^{\otimes n}):=K[t]^{k+1}$ with $\tau$-action given by
\[  \tau\begin{pmatrix}
f_0\\ f_1\\ \vdots \\ f_k \end{pmatrix} = 
\begin{pmatrix}
t-\theta & 1 & 0 &  \cdots & 0 \\
0 & t-\theta &  \ddots &  \ddots &  \vdots \\
\vdots &  \ddots& \ddots & \ddots & 0 \\
\vdots & &  \ddots& \ddots & 1\\
0 & \cdots& \cdots & 0&  t-\theta 
\end{pmatrix}^n \cdot   \begin{pmatrix}
f_0^\tau\\ f_1^\tau\\ \vdots \\ f_k^\tau \end{pmatrix}
.\]
Its rigid analytic trivialization is given by
\[   \Upsilon = \rho_{[k]}(\omega^{-n})=  \begin{pmatrix} \omega & \hd{1}{\omega} & \cdots & \hd{k}{\omega} \\ 0 & \omega & \ddots & \vdots \\ \vdots & \ddots & \ddots & \hd{1}{\omega} \\ 0 & \cdots & 0 & \omega \end{pmatrix}^{-n}
. \] 
\end{prop}

\begin{proof}
This follows from the general description in Section \ref{sec:prolongations-of-t-motives}. One just has to
recognize that $\rho_{[k]}(t-\theta)$ is just the matrix
\[ \begin{pmatrix}
t-\theta & 1 & 0 &  \cdots & 0 \\
0 & t-\theta &  \ddots &  \ddots &  \vdots \\
\vdots &  \ddots& \ddots & \ddots & 0 \\
\vdots & &  \ddots& \ddots & 1\\
0 & \cdots& \cdots & 0&  t-\theta 
\end{pmatrix}. \]
%
%
\end{proof}

\begin{prop}\label{prop:prolong-dual-tensor-power}
The $k$-th prolongation of the dual motive $\mathfrak{C}^{\otimes n}$ is the $\bar{K}[t]$-module
$\rho_k(\mathfrak{C}^{\otimes n}):=\bar{K}[t]^{k+1}$ with $\sigma$-action given by
\[  \sigma\begin{pmatrix}
f_0\\ f_1\\ \vdots \\ f_k \end{pmatrix} = 
\begin{pmatrix}
t-\theta & 1 & 0 &  \cdots & 0 \\
0 & t-\theta &  \ddots &  \ddots &  \vdots \\
\vdots &  \ddots& \ddots & \ddots & 0 \\
\vdots & &  \ddots& \ddots & 1\\
0 & \cdots& \cdots & 0&  t-\theta 
\end{pmatrix}^n \cdot   \begin{pmatrix}
f_0^\sigma\\ f_1^\sigma\\ \vdots \\ f_k^\sigma \end{pmatrix}
.\]
Its rigid analytic trivialization is given by
\[   \Psi = \rho_{[k]}(\Omega^{n}) =  \begin{pmatrix}
\Omega^n & \hd{1}{\Omega^n}  & \hd{2}{\Omega^n} & \cdots&\hd{k}{\Omega^n} \\
0 & \Omega^n & \hd{1}{\Omega^n} &  \ddots   & \vdots \\ 
\vdots &\ddots  & \ddots & \ddots &   \hd{2}{\Omega^n}\\
\vdots & & \ddots  & \Omega^n &  \hd{1}{\Omega^n} \\
0 &  \cdots & \cdots & 0 &  \Omega^n
\end{pmatrix}.  \]
\end{prop}

For the description of the corresponding $t$-modules we restrict to the
prolongations of the Carlitz module, and let the descriptions for the tensor powers as
an exercise for the reader.

\begin{prop}
The $k$-th prolongation $(\rho_kC,\rho_k\phi)$ of the Carlitz module is the $t$-module of dimension
$k+1$ with
\[  \rho_k\phi:\FF_q[t]\to \Mat_{(k+1)\times(k+1)}(K)\{\tau\} \]
given by
\[   ( \rho_k\phi)_t = 
 \begin{pmatrix}
\theta & 0 & \cdots &\cdots & 0 \\
-1 &\theta  & \ddots && \vdots \\
0 & \ddots & \ddots & \ddots& \vdots \\
 \vdots  & \ddots & \ddots & \ddots& 0\\
 0&\cdots &0 & -1 & \theta
\end{pmatrix}+ \one_{k+1}\cdot \tau. \]
\end{prop}

\begin{proof}
This follows from the general description in Section \ref{sec:prolongations-of-t-modules}.
\end{proof}

\section{Hypertranscendence of the Anderson-Thakur function}\label{sec:omega-hypertranscendental}

In this section, we show that the Anderson-Thakur function $\omega$ is hypertranscendental, i.e.~that $\omega$ and all its hyperderivatives $\hd{n}{\omega}$ ($n>0$) are algebraically independent over the field $\bar{K}(t)$. This fact is also given by F.~Pellarin in \cite[Prop.~27]{fp:vscce} by different methods.

We first recall a fact about the evaluations of the Anderson-Thakur function $\omega$ and its hyperderivatives at roots of unity given in \cite{ba-fp:ugtsls} and \cite{am-rp:iddbcppte}.
The evaluation of $\hd{n}{\omega}$ at $t=\zeta$ will be shortly denoted by $\hd{n}{\omega}(\zeta)$.

Moreover, in this section, $K$ will denote the field $\FF_q(\theta)$.

\begin{thm}
Let $\zeta\in \bar{\FF}_q$, let $\pfrak\in \FF_q[t]$ be the minimal polynomial
of $\zeta$, and let $d=\deg(\pfrak)$ be its degree.

For $n\geq 0$, the Carlitz $\pfrak^{n+1}$-torsion extension of $K(\zeta)$ is generated by $\hd{n}{\omega}(\zeta)$, i.e.~
\[    K(\zeta)(C[\pfrak^{n+1}])=K(\zeta,\hd{n}{\omega}(\zeta)). \]
The minimal polynomial of $\omega(\zeta)$ over $K(\zeta)$ is given by
\[   X^{q^d-1}- \beta(\zeta) \in K(\zeta)[X], \] 
where $\beta(t)=\prod_{h=0}^{d-1} (t-\theta^{q^h})\in K[t]\subseteq \TT$.

For $n\geq 1$, the minimal polynomial of $\hd{n}{\omega}(\zeta)$ over $K(\zeta)(C[\pfrak^{n}])$ is given by
\[  X^{q^d}-\beta(\zeta)X-\xi_n(\zeta)\in K(\zeta)(C[\pfrak^{n}])[X], \]
 where
\[ \xi_n(t)= \sum_{l=1}^n \hd{l}{\beta}\cdot \hd{n-l}{\omega}\in \TT. \]
\end{thm}

\begin{proof}
The first part is shown in \cite[Thm.~3.3]{ba-fp:ugtsls} where also the minimal polynomials occur.
The minimality of these polynomials, however, is shown in \cite[Thm.~3.8 \& Rem.~3.9]{am-rp:iddbcppte}.
\end{proof}

\begin{thm}\label{thm:hypertranscendence}
The Anderson-Thakur function $\omega(t)$ is hypertranscendental over $\bar{K}(t)$, i.e.~the set $\{\hd{n}{\omega} \mid n\geq 0\}$ is algebraically independent over $\bar{K}(t)$.
\end{thm}

\begin{proof}
Since $\bar{K}(t)$ is algebraic over $K(t)$, it suffices to show algebraic independence over $K(t)$.
Now, assume for the contrary, that $\omega$ and its hyperderivatives satisfy some algebraic relation. Choose $n$ minimal such that $\omega, \hd{1}{\omega},\ldots, \hd{n}{\omega}$ are algebraically dependent, and choose a polynomial
$0\ne F(X_0,\ldots, X_n)\in K(t)[X_0,\ldots,X_n]$ such that $F(\omega, \hd{1}{\omega},\ldots, \hd{n}{\omega})=0$. Write $F=\sum_{j=0}^k f_jX_n^j$ with $f_j\in K(t)[X_0,\ldots,X_{n-1}]$ and $f_k\ne 0$. After rescaling we can even assume that the coefficients of the $f_j$ are polynomials in $t$, i.e.~$f_j\in K[t][X_0,\ldots,X_{n-1}]$.\\
As we have chosen $n$ to be minimal, and as $f_k\ne 0$, we also have 
\[ f_k(\omega,\hd{1}{\omega},\ldots,\hd{n-1}{\omega})\ne 0\in \TT. \]

Since every nonzero element of $\TT$ has only finitely many zeros in the closed unit disc, for almost all $\zeta\in \bar{\FF}_q^{\,\times}$ we have: $f_k(\omega, \hd{1}{\omega},\ldots, \hd{n-1}{\omega})|_{t=\zeta}\ne 0\in   \CC_\infty$. Hence, for such $\zeta$, $\hd{n}{\omega}(\zeta)$ is a root of the nonzero polynomial
\[ \sum_{j=0}^kf_j(\omega, \hd{1}{\omega},\ldots,\hd{n-1}{\omega})|_{t=\zeta}X_n^j \in \CC_\infty[X_n] \]
of degree $k$.

By construction, the coefficients lie in $K(C[\pfrak^n])(\zeta)$ where $\pfrak\in \FF_q[t]$ is the minimal polynomial of $\zeta$ over $ \FF_q$.
By the theorem above, the minimal polynomial of $\hd{n}{\omega}(\zeta)$ over $K(C[\pfrak^n])(\zeta)$ has degree $q^{\deg(\pfrak)}=\#  \FF_q(\zeta)$ (resp.~degree $q^{\deg(\pfrak)}-1$ if $n=0$).

Therefore, if we choose $\zeta$ such that $\#  \FF_q(\zeta)-1 >k$, this leads to a contradiction.
\end{proof}

\section{Algebraic independence of periods}\label{sec:algebraic-independence}

In this section, we prove our main theorem on the algebraic independence of the periods.

\begin{thm}\label{thm:algebraic-independence}
Let $n\in\NN$ be prime to $q$, let $C^{\otimes n}$ be the $n$-th tensor power of the Carlitz module and let
\[ \begin{pmatrix} z_1 \\ \vdots \\ z_n \end{pmatrix} \in \CC_\infty^n\]
be a generator for the period lattice. Then $z_1,z_2,\ldots, z_n$ are algebraically independent over $\bar{K}$.
\end{thm}

\begin{rem}
As already noted in \cite{ga-dt:tpcmzv}, if $n$ is a power of the characteristic $p=\ch(\FF_q)$,
then all but the last coordinate are $0$. We will make a precise statement in the case that $p$ divides $n$ at the end of this section.
\end{rem}

For proving the theorem, we first give a formula for these coordinates using 
evaluations of hyperderivatives.

\begin{lem}\label{lem:description-of-coordinates}
Let the generator above be chosen such that $z_n=\tilde{\pi}^n$.
Then the coordinates $z_1,z_2,\ldots, z_n$ fulfill the equalities
\[ z_i= (-1)^n \hd{n-i}{\phantom{\Big( \!\!\!} (t-\theta)^n\omega(t)^{n}}|_{t=\theta}, \]
i.e.~$z_i$ is the $(n-i)$-th hyperderivative of the function $(\theta-t)^n\omega(t)^{n}$ evaluated
at $t=\theta$.
\end{lem}

\begin{proof}
As $\omega$ has a pole of order $1$ at $t=\theta$, $\omega^n$ has a pole of order $n$.
Building on work of Anderson and Thakur \cite[\S 2.5]{ga-dt:tpcmzv}, we write
$\omega^n$ as a Laurent series in $(t-\theta)$,
\[ \omega^n=\sum_{j=-n}^\infty c_j (t-\theta)^j\in \CC_\infty\ls{t-\theta}. \]
Then the coordinates are explicitly given by
\[ z_i= (-1)^n c_{-i} \]
for $i=1,\ldots, n$ (see \cite[Cor.~2.5.8]{ga-dt:tpcmzv}, and be aware that $\bar{\pi}$ ibid. equals $-\pitilde$).

On the other hand, for any $0\leq k\leq n$:
\begin{eqnarray*}
\hd{k}{\phantom{\Big( \!\!\!} (t-\theta)^n\omega(t)^{n}} &=&
\hd{k}{ \sum_{j=-n}^\infty c_j (t-\theta)^{j+n}} \\
&=& \sum_{j=k-n}^\infty c_j \binom{j+n}{k} (t-\theta)^{j+n-k}.
\end{eqnarray*}
Hence for $i=1,\ldots, n$:
\begin{eqnarray*}
   (-1)^n \hd{n-i}{\phantom{\Big( \!\!\!} (t-\theta)^n\omega(t)^{n}}|_{t=\theta}
&=&  (-1)^n  \sum_{j=-i}^\infty c_j \binom{j+n}{n-i} (t-\theta)^{j+i}|_{t=\theta} \\
&=&  (-1)^n c_{-i} =z_i. 
\end{eqnarray*}
\end{proof}

A second ingredient is a relation between hyperderivatives of functions and hyperderivatives of
powers of that function.

\begin{lem}\label{lem:finite-extension}
Let $0\neq f\in \laurent$, $k\geq 0$ and let $E$ be the field extension of $K(t)$ generated by the entries of $\rho_{[k]}(f)$, i.e.~generated by $f,\hd{1}{f},\ldots, \hd{k}{f}$.
For $n\in \NN$ prime to $q$, let $F$ be the field extension of $K(t)$ generated by the entries of $\rho_{[k]}(f^n)$, i.e.~generated by $f^n,\hd{1}{f^n},\ldots, \hd{k}{f^n}$.
Then $E$ is generated over $F$ by $f$, and in particular, $E$ is finite algebraic over $F$.
\end{lem}

\begin{proof}
We only have to show that $\hd{j}{f}\in F(f)$ for $1\leq j\leq k$.
Let $p=\ch(\FF_q)$, and $s\in \NN$ such that $p^s>k$. Then for 
all $j$ not divisible by $p^s$, one has $\hd{j}{f^{p^s}}=0$, since
\[ \mathcal{D}(f^{p^s})=\mathcal{D}(f)^{p^s} \]
is a power series in $X^{p^s}$. In particular, we have  $\hd{j}{f^{p^s}}=0$ for all 
$1\leq j\leq k$, and therefore $\rho_{[k]}(f^{p^s})$ is the scalar matrix with diagonal
entries equal to $f^{p^s}$.

As $n$ was prime to $q$, and hence prime to $p$, there are $a,b\in \ZZ$ such that
$ap^s+bn=1$.
Therefore,
\[ \rho_{[k]}(f)=\rho_{[k]}(f^{ap^s+bn})
= \rho_{[k]}(f^{p^s})^a\cdot \rho_{[k]}(f^n)^b
=(f^{p^s})^a\cdot \rho_{[k]}(f^n)^b \]
has entries in $F(f)$.
\end{proof}

\begin{proof}[Proof of Thm.~\ref{thm:algebraic-independence}]
By Prop.~\ref{prop:prolong-dual-tensor-power},
a rigid analytic trivialization of $\rho_{n-1}(\mathfrak{C}^{\otimes n})$,
the $(n-1)$-th prolongation of $\mathfrak{C}^{\otimes n}$, is given by the matrix
\[ \rho_{[n-1]}(\Omega^n)=  \begin{pmatrix}
\Omega^n & \hd{1}{\Omega^n}  & \hd{2}{\Omega^n} & \cdots&\hd{n-1}{\Omega^n} \\
0 & \Omega^n & \hd{1}{\Omega^n} &  \ddots   & \vdots \\ 
\vdots &\ddots  & \ddots & \ddots &   \hd{2}{\Omega^n}\\
\vdots & & \ddots  & \Omega^n &  \hd{1}{\Omega^n} \\
0 &  \cdots & \cdots & 0 &  \Omega^n
\end{pmatrix}.  \]
Let $F$ be the field generated by the entries of $\rho_{[n-1]}(\Omega^n)$ over $\bar{K}(t)$.
As,
\[ \rho_{[n-1]}(\Omega^n)= \rho_{[n-1]}\left( (t-\theta)^{-n}\omega(t)^{-n}\right)
= \left( \rho_{[n-1]} (t-\theta)\right)^{-n} \cdot  \left(\rho_{[n-1]}(\omega^n)\right)^{-1} \]
and $ \rho_{[n-1]} (t-\theta)\in \GL_n(K(t))$, the field $F$ is also generated over $\bar{K}(t)$ by
the entries of $\rho_{[n-1]}(\omega^{n})$, and in particular is a subfield of finite index of the field generated by the entries of $\rho_{[n-1]}(\omega)$, as shown in Lemma \ref{lem:finite-extension}.
Since $\omega$ is hypertranscendental (see Theorem \ref{thm:hypertranscendence}), 
the latter has transcendence degree over $\bar{K}(t)$ equal to $n$. Hence, the field $F$ has transcendence degree $n$ over $\bar{K}(t)$.

Let $L$ be the field extension of $\bar{K}$ generated by the entries of 
$\rho_{[n-1]}(\Omega^n)|_{t=\theta}$.
Then by the proof of \cite[Thm. 5.2.2]{mp:tdadmaicl} (see Thm.~\ref{thm:conseq-of-abp}), the transcendence degree of $L/\bar{K}$ is the same as the transcendence degree of $F/\bar{K}(t)$, i.e.~equals $n$. 
On the other hand, $L$ is also generated as a field by the entries of the inverse
of $\rho_{[n-1]}(\Omega^n)|_{t=\theta}$, and using Lemma \ref{lem:description-of-coordinates}, we get
\begin{eqnarray*}
\left(\rho_{[n-1]}(\Omega^n)|_{t=\theta}\right)^{-1}
&=& \left( \rho_{[n-1]}(\Omega^n)^{-1}\right)|_{t=\theta}
=\rho_{[n-1]}(\Omega^{-n})|_{t=\theta} \\
&=& \rho_{[n-1]}((t-\theta)^n\omega^n)|_{t=\theta}\\
&=& (-1)^n \cdot\begin{pmatrix}
  z_n &   z_{n-1}  &   z_{n-2} & \cdots &  z_1 \\
0 &   z_n &    z_{n-1} &  \ddots   & \vdots \\ 
\vdots &\ddots  & \ddots & \ddots &     z_{n-2}\\
\vdots & & \ddots  &    z_n &    z_{n-1} \\
0 &  \cdots & \cdots & 0 &    z_n
\end{pmatrix}.
\end{eqnarray*}
Hence, $z_1,\ldots, z_n$ are algebraically independent over $\bar{K}$.
\end{proof}

In the case that the characteristic $p$ divides $n$, we can also make a precise statement
on the algebraic independence.

\begin{cor}
Let $n\in\NN$ be arbitrary, let $C^{\otimes n}$ be the $n$-th tensor power of the Carlitz module and let
\[ \begin{pmatrix} z_1 \\ \vdots \\ z_n \end{pmatrix} \in \CC_\infty^n\]
be the generator for the period lattice with $z_n=\pitilde^n$. If $p^s$ is the exact power of
$p$ dividing $n$, then $z_i\neq 0$ precisely, when $p^s$ divides $i$, and all nonzero coordinates are algebraically independent over $\bar{K}$.
\end{cor}

\begin{proof}
The hyperdifferential operators on $\laurent$ satisfy
\[   \hd{i}{f^{p^s}}=\left\{ \begin{array}{cl}
 0 & \text{if }p^s\text{ does not divide } i\\
\left(\hd{i/p^s}{f}\right)^{p^s} & \text{if }p^s\text{ divides }i,
\end{array}\right.
\]
for all $f\in \laurent$, as one readily sees by using the homomorphism $\mathcal{D}$.

Applying this to $f=\Omega^{n/p^s}$, we see that the nonzero entries in
$\rho_{[n-1]}(\Omega^n)$ are the $\hd{i}{\Omega^n}$ with $p^s$ divides $i$ and those are
equal to $\left(\hd{i/p^s}{\Omega^{n/p^s}}\right)^{p^s}$.

By specializing the inverse of $\rho_{[n-1]}(\Omega^n)$ to $t=\theta$ as in the proof of Theorem \ref{thm:algebraic-independence}, we see that the coordinates $z_i$ where $p^s$ does not divide $i$ are equal to zero, and that the other coordinates are just the $p^s$-powers of the coordinates of a period lattice generator for the $n/p^s$-th tensor power of the Carlitz module.
Hence, by Theorem \ref{thm:algebraic-independence}, they are algebraically independent over $\bar{K}$.
\end{proof}

\bibliographystyle{plain}
\def\cprime{$'$}


\vspace*{.5cm}

\parindent0cm

\end{document}